\documentclass[11pt,reqno,a4paper]{amsart}
\usepackage{color,enumerate,cancel,hyperref,mathtools,tikz-cd,amssymb,amsfonts}
\newtheorem{theorem}{Theorem}
\newtheorem*{theorem*}{Theorem}

\newtheorem{proposition}[theorem]{Proposition}
\newtheorem{lemma}[theorem]{Lemma}
\newtheorem{corollary}[theorem]{Corollary}

\theoremstyle{remark}\newtheorem*{case}{Case}
\theoremstyle{definition}

\theoremstyle{remark}
\newtheorem{remark}[theorem]{Remark}

\newcommand{\Lip}{\operatorname{Lip}_0}
\newcommand{\Lipp}{\operatorname{Lip}}
\newcommand{\Span}{\operatorname{span}}
\newcommand{\Id}{\operatorname{Id}}
\newcommand{\R}{\mathbb{R}}
\newcommand{\N}{\mathbb{N}}

\newcommand{\F}{\mathcal{F}}

\begin{document}
\title[On large $\ell_1$-sums of Lipschitz-free spaces and applications]{On large $\ell_1$-sums of Lipschitz-free spaces and applications}
\author[L. Candido]{Leandro Candido}
\address{Universidade Federal de S\~ao Paulo - UNIFESP. Instituto de Ci\^encia e Tecnologia. Departamento de Matem\'atica. S\~ao Jos\'e dos Campos - SP, Brasil}
\email{leandro.candido@unifesp.br}
\author[H. H. T. Guzm\'an]{H\'ector H. T. Guzm\'an}
\address{Universidade Federal de S\~ao Paulo - UNIFESP. Instituto de Ci\^encia e Tecnologia. Departamento de Matem\'atica. S\~ao Jos\'e dos Campos - SP, Brasil}
\email{torres.hector@unifesp.br}



\begin{abstract}
We prove that the Lipschitz-free space over a Banach space $X$ of density $\kappa$, denoted by $\mathcal{F}(X)$, is linearly isomorphic to its $\ell_1$-sum $\left(\bigoplus_{\kappa}\mathcal{F}(X)\right)_{\ell_1}$. This provides an extension of a previous result from Kaufmann in the context of non-separable Banach spaces. Further, we obtain a complete classification of the spaces of real-valued  Lipschitz functions that vanish at $0$ over a $\mathcal{L}_p$-space. More precisely, we establish that, for every $1\leq p\leq \infty$, if $X$ is a $\mathcal{L}_p$-space of density $\kappa$, then $\mathrm{Lip}_0(X)$ 
is either isomorphic to $\mathrm{Lip}_0(\ell_p(\kappa))$ if $p<\infty$, or $\mathrm{Lip}_0(c_0(\kappa))$ if $p=\infty$.
\end{abstract}

\maketitle

\section{Introduction}
In this paper we are mainly interested in investigating the geometry of Lipschitz-free spaces over a Banach space $X$, $\F(X)$, and its 
topological dual, the space of Lipschitz functions that vanish at $0$, $\Lip(X)$. The geometry of these spaces always pose many intriguing questions and so far very little is known. For example, it is an open problem whether, for distinct integers $n$ and $m$ greater or equal to $2$, $\F(\R^n)$ is linearly isomorphic to $\F(\R^m)$ (thanks to Naor and Schechtman \cite{NaorSchet} we know that $\F(\R)$ is not isomorphic to $\F(\R^2)$). In the infinite-dimensional setting, in a recent paper \cite{ANPP} the authors have shown that $\F(\ell_p)$ is not linearly isomorphic to $\F(c_0)$ whenever $1<p<\infty$. However, it still not known whether $\F(\ell_1)$ is isomorphic to $\F(c_0)$. Similar questions in the context of Lipschitz functions spaces seems to be even more difficult to address as a plethora of mutually non-isomorphic Banach spaces may have isomorphic duals and, as far as we know, there are still no examples of infinite-dimensional Banach spaces $X$ and $Y$ of the same density such that $\Lip(X)$ is not isomorphic to $\Lip(Y)$. 

Ideas and techniques developed years ago by N. J. Kalton \cite{Kalton} have proven to be invaluable in this field, see for example \cite{AACD}, \cite{ANPP}, \cite{CCD}, \cite{Kauf}. By applying some of Kalton's results \cite[\S 4]{Kalton}, P. L. Kauffman \cite{Kauf} established the striking fact that, for all Banach spaces $X$, $\F(X)$ is linearly isomorphic to its $\ell_1$-sum. In symbols,
\[\F(X)\sim \left(\bigoplus_{\N}\F(X)\right)_{\ell_1}.\]
Furthermore, for all Banach spaces $X$, $\F(X)$ is linearly isomorphic to $\F(B_X)$, where $B_X$ denotes the unit ball of $X$. These theorems are now part of the basic toolkit in any study of this kind and we refer to \cite{AACD} for a deep investigation on this topic.

Our contribution to this field in the present paper can be divided into two interrelated parts. In the first part we obtain a generalization of the aforementioned Kaufmann's $\ell_1$-sum theorem in the specific context of non-separable Banach spaces. Specifically, our first main result can be stated as follows.

\begin{theorem}\label{ell_1Sum}
If $X$ is a Banach space of density $\kappa$, then
\[\F(X)\sim \left(\bigoplus_{\kappa}\F(X)\right)_{\ell_1}.\]
\end{theorem}

In the second part we turn our attention to the $\Lip(X)$ spaces. Our second main result consists of a non-separable version  \cite[Lemma 1.3]{CCD}.

\begin{theorem}\label{UltraLIP} Let $M$ be a pointed metric space with an origin $0$. Let $\varGamma$ be an infinite collection of subsets of $M$ directed by upward inclusion, each of it containing $0$ and such that $\bigcup_{N\in \varGamma}N$ is dense in $M$. Then 
\[\Lip(M) \stackrel{c}{\hookrightarrow} \left(\bigoplus_{N\in \varGamma} \Lip(N)\right)_{\ell_{\infty}}.\]
\end{theorem}

As an application of the previous theorems, we extend \cite[Theorem 3.3]{CCD} by giving a complete classification of the Banach spaces 
$\Lip(X)$ where $X$ denotes an $\mathcal{L}_p$-space, see \cite[Definition II.5.2]{LinTza}. 

\begin{theorem}\label{Lpspaces}
For some $1\leq p \leq \infty$, let $X$ be a $\mathcal{L}_p$-space of density $\kappa$. Then $\Lip(X)$ is linearly isomorphic to either 
$\Lip(\ell_p(\kappa))$, if $p<\infty$, or  $\Lip(c_0(\kappa))$, if $p=\infty$.
\end{theorem}

Theorem \ref{Lpspaces} provides a complete classification of the Banach spaces $\Lip(C(K))$, where $K$ stands for an infinite Hausdorff compact space, since is a well-known fact that every $C(K)$ is a $\mathcal{L}_\infty$-space, see \cite[p. 198-199]{LinTza}. The next corollary extends the main result of \cite{CanKau} and also the main result of \cite{DutFec} in the dual setting.

\begin{corollary}\label{LipC(K)}
If $K$ is an infinite Hausdorff compact space of weight $\kappa$, then $\Lip(C(K))$ is linearly isomorphic to $\Lip(c_0(\kappa))$.
\end{corollary}

The paper is organized as follows. In Section \ref{Sec-Term} we set the main terminology to be used in this paper. In Section \ref{Sec-Aux} we present some necessary basic propositions. In Section \ref{Sec-ell1} we establish our first main result, Theorem \ref{ell_1Sum}. Finally, in Section \ref{Sec-ellinfty} we prove Theorems \ref{UltraLIP} and \ref{Lpspaces}.


\section{Terminology and Preliminaries}
\label{Sec-Term}
Given metric spaces $(M,d)$ and $(N,e)$, for every function $f:M\to N$ we may consider its Lipschitz number 
\[L(f) = \displaystyle \sup_{x\neq y \in M} \displaystyle \frac{e(f(x),f(y))}{d(x,y)}.\]
We say that  $f$ is a Lipschitz function if $L(f)<\infty$ and in this case $L(f)$ is the smallest of all numbers $\lambda \geq 0$ satisfying, for all $x,y \in M$, the formula 
\[e(f(x),f(y)) \leq \lambda d(x,y).\] 

If $M$ is a pointed metric space, that is, a metric space with a distinguished point $0$ called origin, then $\Lip(M)$ denotes 
the Banach space of all Lipschitz functions $f:M\to \R$ that vanish at $0$, endowed with the norm $\|f\|_{\Lipp}=L(f)$.

The Banach space $\Lip(M)$ is a dual space. Its canonical predual, called Lipschitz-free space or Arens-Eells space and denoted $\F(M)$,  
can be obtained as closed linear span of all the evaluation functionals in $\Lip(M)^*$. We refer the reader to \cite{God} and \cite{Weaver} for more details about these spaces.

If $K$ is a compact Hausdorff space, $C(K)$ stands for the Banach space of all continuous functions $f:K\to \R$, equipped with the norm \[\|f\|_{\infty}=\sup_{k\in K}|f(k)|.\]
We refer to Semadeni's book \cite{Se} for a complete survey on $C(K)$ spaces. 

Banach spaces $X$ and $Y$ are said to be isomorphic if there exists a bijective bounded linear operator $T:X\to Y$. In this case, the number $\|T\|\|T^{-1}\|$ will be called distortion of $T$. We will often write $X\sim Y$ to indicate that there is an isomorphism $T:X\to Y$ and $X\stackrel{\lambda}{\sim} Y$ to emphasize that $\|T\|\|T^{-1}\|\leq \lambda$. We will write $Y\stackrel{c}{\hookrightarrow} X$ to indicate that $Y$ is isomorphic to a complemented subspace of $X$. This means that there are bounded linear operators $T:X\to Y$ and $S:Y\to X$ so that $S\circ T=\Id_X$
is the identity operator. 

Given $1 \leq  p \leq \infty$, a Banach space $X$ is said to be a $\mathcal L_{p}$-space (or a $\mathcal L_{p,\lambda}$-space) if there is $\lambda\geq 1$ such that every finite-dimensional subspace $Y$ of $X$ is contained in a subspace $Z$ of $X$ with $Z\stackrel{\lambda}{\sim} \ell_p^n$ for some $n \in \N$, see \cite[Definition II.5.2]{LinTza}.

For a pointed metric space $M$ and a subset $N\subset M$ containing the origin $0$, we denote by 
$\mathrm{Ext}(N,M)$ the set of all extension operators $E:\Lip(N)\to \Lip(M)$, that is, $E$ is a bounded linear map such that 
$E(f)\restriction_N=f$ for each $f\in \Lip(N)$. We denote by $\mathrm{Ext}^{pt}(N,M)$ the subset of $\mathrm{Ext}(N,M)$
comprising all the extension operators which are pointwise-pointwise-continuous on the bounded subsets of $\Lip(N)$.

Whenever $\mathrm{Ext}(N,M)\neq \emptyset$, then $\Lip(N)$ is isomorphic to a complemented subspace of  $\Lip(M)$. Indeed, if $R:\Lip(M)\to \Lip(N)$ is the natural restriction $R(g)=g\restriction_N$, then it is immediate that $R\circ E=\Id_{\Lip(N)}$ is the identity operator. If in particular $E\in \mathrm{Ext}^{pt}(N,M)$, then $E$ is a dual operator since it is weak$^*$-weak$^*$-continuous, see \cite[Theorem 3.3]{Weaver}. In this case we obtain the stronger conclusion that $\F(N)$ is isomorphic to a complemented subspace of $\F(M)$.

For a Banach space $X$ we denote by $B[x,r]=x+rB_X$ ($S[x,r]=x+rS_X$) the closed ball (sphere) with center $x$ and radius $r$. A collection $\mathcal{A} \subset X$ will be called a packing (or $r$-packing) in the unit ball $B_X$, if there is $r>0$ so that all the members of the family $\{B[a,r]:a \in \mathcal{A}\}$ are contained in $B_X$ and have pairwise disjoint interiors. 

The concept of convergence with respect to ultrafilters will also play an important role in this research. Given an ultrafilter $\mathcal{U}$, we say that an indexed family of real numbers $(a_i)_{i\in \varGamma}$ converges to an ultralimit $a\in\R$ with respect to $\mathcal{U}$, and we write $\lim_{\mathcal{U}}a_i=a$, if $\{i\in \varGamma: |a_i-a|<\epsilon\}\in \mathcal{U}$ for every $\epsilon>0$. It is important to observe that if $(a_i)_{j\in \varGamma}$ is bounded family in $\R$, then a ultralimit $\lim_{\mathcal{U}}a_i$ exists in $\R$ for every ultrafilter $\mathcal{U}$ on $\varGamma$, see \cite[Theorem 3.52]{HindStrauss}. We refer the reader to \cite[\S 3.5]{HindStrauss} for a detailed presentation on ultralimits and its properties.

\begin{remark}\label{Ultralimits}If $\varGamma$ is a directed set, any free ultrafilter containing all the members of the collection $\mathcal{G}=\{G_j:j\in \varGamma\}$, where $G_j=\{i\geq j:i \in \varGamma\}$ for each $j\in \varGamma$, will be called directed ultrafilter. Whenever $\varGamma$ is infinite, the collection $\mathcal{G}\cup\{\varGamma\setminus\{j\}:j\in \varGamma\}$
has the finite intersection property and therefore, by the Kuratowski--Zorn Lemma, it is contained in an ultrafilter $\mathcal{U}$ which, in turn, will be directed. If  $\mathcal{U}$ is a directed ultrafilter, then for every bounded net $(a_i)_{i\in \varGamma}$ and every $j\in \varGamma$, there is a sequence $(i_n)_{n\in \N}$ in $\varGamma$ such that $j<i_1<i_2<i_3<\ldots$ and \[\lim_{n\to\infty}a_{i_n}=\lim_{\mathcal{U}}a_i.\]
Indeed, since $(a_i)_{i\in \varGamma}$ is bounded $a=\lim_{\mathcal{U}}a_i$ exists and it is a real number. We fix $i_0=j$ and for $n\in \N$ arbitrary, assuming that $i_{n-1}$ is defined, we fix $i_{n}\in (G_{i_{n-1}}\setminus \{i_{n-1}\})\cap \{i\in \varGamma:|a_i-a|<\frac{1}{n}\}$. It is readily seen that $(i_n)_{n\in \N}$ has the required properties.
\end{remark}


\section{Auxiliary Results}
\label{Sec-Aux}

In this section we present some elementary propositions necessary for Lemma \ref{Extension} and our first main result, Theorem \ref{ell_1Sum}. 

\begin{proposition}\label{RadialRetraction}There is a $2$-Lipschitz retraction from a Banach space $X$ to any closed ball $B[a,r]$.
\end{proposition}
\begin{proof}
It is an elementary checking that $h:X\to B[a,r]$ given by the formula
\[ h(x) = \left\{ \begin{array}{cc}
  x & \text{ if } \|x-a\|\leq r\\
  a+r\frac{x-a}{\|x-a\|} & \text{ if } \|x-a\|>r  
\end{array} \right. \]
defines a $2-$Lipschitz retraction.
\end{proof}

\begin{proposition}\label{UmaBola}
Let $X$ be a Banach space, $0<r<1$ be a real number and $a\in B_X$ be such that $\|a\|>r$. Then there is an isomorphism
$T:\F(B_X)\oplus_1 \R\to \F(\{0\}\cup B[a,r])$ with distortion $\|T\|\|T^{-1}\|\leq \max\{\frac{1}{r},\frac{2}{\|a\|-r}\}$.
\end{proposition}
\begin{proof}
Consider the map $T:\Lip(\{0\}\cup B[a,r])\to \Lip(B_X)\oplus_{\infty} \R$ given by
\[T(f)(x)=(f(rx+a)-f(a),f(a)).\]

It is evident that $T$ is a well defined linear map and, since
\[\|T(f)\|\leq \max\{L(f\restriction_{B[a,r]}),(|f(a)|/\|a\|)\}\leq\max\{L(f\restriction_{B[a,r]}),\sup_{x\in B[a,r]}(|f(x)|/\|x\|)\}=\|f\|_{\Lipp}\] 
for every $f\in \Lip(\{0\}\cup B[a,r])$, $T$ is continuous with $\|T\|\leq 1$. 
Furthermore, $T$ is a bijection and the inverse map $S:\Lip(B_X)\oplus_\infty \R\to \Lip(\{0\}\cup B[a,r])$ 
can be explicitly described by the formula
\begin{displaymath}
S(g,\lambda)(x)=\left\{\begin{array}{ll}
g((x-a)/r)+\lambda& \text{ if }x\in B[a,r]\\
0&\text{ if } x=0\\
\end{array} \right.
\end{displaymath}
Note that, for every $(g,\lambda)\in \Lip(B_X)\oplus_{\infty} \R$ and for all $x,y\in B[a,r]$,
\[|S(g,\lambda)(x)-S(g,\lambda)(y)|\leq \frac{1}{r}\|(g,\lambda)\|\|x-y\|\]
and
\begin{align*}|S(g,\lambda)(x)-S(g,\lambda)(0)|&=\left|g((x-a)/r)+\lambda\right|\leq \|g\|_{\Lipp}+|\lambda|\leq 2\|(g,\lambda)\|\\
&=\frac{2}{\|a\|-r}\|(g,\lambda)\|(\|a\|-r)\leq \frac{2}{\|a\|-r}\|(g,\lambda)\|\|x-0\|.
\end{align*}
We deduce that $\|S\|\leq \max\{\frac{1}{r},\frac{2}{\|a\|-r}\}$. It is easily seen that $T$ is pointwise-pointwise-continuous on the bounded subsets of $\Lip(\{0\}\cup B[a,r])$ from where the thesis follows.\\
\end{proof}

\begin{proposition}\label{SeparadorDeBolas}
Let $(M,d)$ be a bounded metric space with a fixed distinguished point $0$. Let $\{B_i:i\in I\}$ be a family of nonempty subsets of $M\setminus \{0\}$ such that $d(B_i,B_j)>\delta>0$ whenever $i\neq j$. If $N= \bigcup_{i\in I}(\{0\}\cup B_i)$, then
\[\F(N)\sim \left(\bigoplus_{i\in I}\F(\{0\}\cup B_i)\right)_{\ell_1}.\]
\end{proposition}
\begin{proof}
Let $n\in \N$ be so that $n\delta\geq 2 \mathrm{diam}(M)$. Then, for every $x\in B_i$ and $y\in B_j$, if $i\neq j$, we have 
$d(x,0)+d(y,0)\leq 2\mathrm{diam}(M)\leq n\delta \leq n d(x,y)$. We are done by applying \cite[Proposition 6]{CanKau}.
\end{proof}


\section{On large $\ell_1$-sums of Lipschitz-free spaces}
\label{Sec-ell1}

The next lemma, which is crucial for Theorem \ref{ell_1Sum}, allows us to obtain an extension operator from a $r$-packing to the whole unit ball $B_X$. 

\begin{lemma}\label{Extension}
In a Banach space $X$, for some $s>0$, assume that $\mathcal{A}$ is a $s$-packing in the unit ball $B_X$, containing the origin.
For each $0<r<s$, there is an extension operator $E\in \mathrm{Ext}^{pt}(N,B_X)$, where $N=\bigcup_{a\in \mathcal{A}}B[a,r]$,
with $\|E\|\leq \frac{4s-2r+1}{s-r}$.
\end{lemma}
\begin{proof}
Consider $M=\bigcup_{a\in \mathcal{A}}B[a,s]$. For each $a\in \mathcal{A}$, 
let $\varphi_a:B[a,r] \cup (B_X\setminus B[a,s])\to \R$ be given by
\begin{displaymath}
\varphi_a(x)=\left\{\begin{array}{ll}
1& \text{ if }x\in B[a,r]\\
0&\text{ otherwise}\\
\end{array} \right.
\end{displaymath}
For every $x\in B[a,r]$ and $y\in B_X\setminus B[a,s]$ we have $\|x-y\|\geq s-r$. Therefore, $\varphi_a$ 
is a Lipschitz function with Lipschitz number $L(\varphi_a)=\frac{1}{s-r}$. 
Let $\gamma_a:B_X\to \R$ be a McShane's extension of $\varphi_a$ to $B_X$, that is, 
\[\gamma_a(x)=\inf\{\varphi_a(y)+\|x-y\|/(s-r): y \in B[a,r]\cup(B_X\setminus B[a,s])\}.\]

For each $a\in \mathcal{A}$, let $h_a:B_X\to B[a,r]$ be the restriction of the radial retraction from Proposition \ref{RadialRetraction}. Given $f\in \Lip(N)$ we define a map $E(f):B_X\to \R$ by the formula 
\begin{displaymath}
E(f)(x)=\left\{\begin{array}{ll}
\gamma_a(x)f(h_a(x))& \text{ if } \|x-a\|\leq s\text{ for some }a\in \mathcal{A}\\
0&\text{ otherwise}\\
\end{array} \right.
\end{displaymath}

Note that $E(f)$ is well defined and $E(f)(x)= f(x)$ for all $x\in N$. To check that $E(f)$ is a Lipschitz function let $x,y\in B_X$, $x\neq y$, be arbitrary. We distinguish $3$ main cases:
\begin{case}[1]$x,y\in B[a,s]$.
\end{case}
\begin{align*}
\frac{|E(f)(x)-E(f)(y)|}{\|x-y\|}&=\frac{|\gamma_a(x)f(h_a(x))-\gamma_a(y)f(h_a(y))|}{\|x-y\|}\\
&\leq\frac{|\gamma_a(x)-\gamma_a(y)|}{\|x-y\|}|f(h_a(x))|+\frac{|f(h_a(y))-f(h_a(x))|}{\|x-y\|}|\gamma_a(y)|\\
&\leq L(\gamma_a)\sup_{z\in B[a,r]}|f(z)|+L(f\circ h_a) \sup_{z\in B[a,s]}|\gamma_a(z)|\\
&\leq \frac{\|f\|_{\Lipp}}{(s-r)}\sup_{z\in B[a,r]}\|z\|+ 2\|f\|_{\Lipp}\sup_{z\in B[a,s]}(1+\frac{\|z-a\|}{(s-r)})\\
&\leq\|f\|_{\Lipp}\left(\frac{4s-2r+1}{s-r}\right).
\end{align*}

\begin{case}[2]
$x\in B[a,s]$ and $y\in B_X\setminus M$.
\end{case}
Let $z\in S[a,s]$ the point in the line segment connecting $x$ and $y$. Since $E(f)(z)=E(f)(y)=0$, we have by the previous case
\begin{align*}
|E(f)(x)-E(f)(y)|&\leq |E(f)(x)-E(f)(z)|+|E(f)(z)-E(f)(y)|\\
&\leq\|f\|_{\Lipp}\left(\frac{4s-2r+1}{s-r}\right)\|x-z\|\leq \|f\|_{\Lipp}\left(\frac{4s-2r+1}{s-r}\right)\|x-y\|.
\end{align*}

\begin{case}[3] $x\in B[a_1,s]$ and $y\in B[a_2,s]$ with $a_1\neq a_2$.
\end{case}
Let $z_1\in S[a_1,s]$ and $z_2\in S[a_2,s]$
be the points in the line segment connecting $x$ and $y$. Since $E(f)(z_1)=E(f)(z_2)=0$, we have by Case (1),
\begin{align*}
|E(f)(x)-E(f)(y)|&\leq |E(f)(x)-E(f)(z_1)|+|E(f)(z_1)-E(f)(z_2)|+|E(f)(z_2)-E(f)(y)|\\
&\leq \|f\|_{\Lipp}\left(\frac{4s-2r+1}{s-r}\right)(\|x-z_1\|+\|z_2-y\|)\leq \|f\|_{\Lipp}\left(\frac{4s-2r+1}{s-r}\right)\|x-y\|.
\end{align*}

Hence $E(f)\in \Lip(B_X)$. It follows that $E:\Lip(N)\to \Lip(B_X)$ is a well defined map, easily seen to be linear with $\|E\|\leq \frac{4s-2r+1}{s-r}$. We check that $E$ is pointwise-pointwise-continuous. Indeed, let $(g_i)_{i\in \varGamma}$ be a bounded net in $\Lip(N)$ converging pointwise to a function $g\in \Lip(N)$. Given $x\in B_X$, if $x\in B_X\setminus M$, then $E(g_i)(x)= E(g)(x)=0$ for all $i\in \varGamma$. If 
$\|x-a\|\leq s$ for some $a\in \mathcal{A}$, then 
\[\lim_{i\to \infty}E(g_i)(x)=\lim_{i\to \infty} \gamma_a(x)g_i(h_a(x))=\gamma_a(x)g(h_a(x))=E(g)(x).\]
And we conclude that $E\in \mathrm{Ext}^{pt}(N,B_X)$ 

\end{proof}

We are now in position of proving our first main result. 

\begin{proof}[Proof of Theorem \ref{ell_1Sum}]Let $X$ be a non-separable Banach space of density $\kappa$ (the separable case follows from \cite[Theorem 3.1]{Kauf}). We construct a $1/4-$packing $\mathcal{A}\subset B_X$ of cardinality $\kappa$ in the following way. We let $x_0$ be an arbitrary point in $S_X$ and given $j<\kappa$ assume that a collection $\{x_i: i<j\}$ was obtained. Then, $Y=\overline{\Span}\{x_i: i<j\}$ is a proper subspace of $X$ and, by the Riesz Lemma, there is $x_j\in S_X$ such that $\|x_j-x_i\|\geq 1/2$ for each $i\leq j$. By induction, we obtain a $1/2$-dispersed collection $\mathcal{C}\subset S_X$ of cardinality $\kappa$. We are done by fixing $\mathcal{A}=\{0\}\cup \{x/2:x\in \mathcal{C}\}$.\\

We let $N=\bigcup_{a\in \mathcal{A}}B[a,1/8]$. Since $\F(B[0,1/8])\stackrel{1}{\sim}\F(B_X)$ and, by  Proposition \ref{UmaBola},  $\F(B[a,1/8])\stackrel{8}{\sim}\F(B_X)\oplus_1 \R$ for each $a\in \mathcal{A}\setminus \{0\}$, by applying Proposition \ref{SeparadorDeBolas} we obtain
\begin{align*}
\F(N)&\sim \left(\bigoplus_{a\in \mathcal{A}\setminus \{0\}}\F(\{0\}\cup B[a,1/8])\right)_{\ell_1} \oplus_1 \F(B[0,1/8])\\
     &\sim \left(\bigoplus_{a\in \mathcal{A}\setminus \{0\}}\F(B_X)\oplus_1 \R\right)_{\ell_1}\oplus_1 \F(B_X)
     \sim \left(\bigoplus_{a\in \mathcal{A}}\F(B_X)\right)_{\ell_1}\oplus_1 \ell_1(\kappa).
\end{align*}

Recalling that $\F(B_X)\oplus_1 \ell_1(\kappa)\sim \F(B_X)$ by \cite[Proposition 3]{HajNov} and $\F(B_X)\sim \F(X)$ due to \cite[Corollary 3.3]{Kauf}, we have
\begin{align*}
\F(N)\sim \left(\bigoplus_{a\in \mathcal{A}}\F(B_X)\right)_{\ell_1}\sim \left(\bigoplus_{a\in \mathcal{A}}\F(X)\right)_{\ell_1}.
\end{align*}

According to Lemma \ref{Extension}, there is an extension operator $E\in \mathrm{Ext}^{pt}(N,B_X)$. Then,
\[\left(\bigoplus_{a\in \mathcal{A}}\F(X)\right)_{\ell_1}\sim \F(N)\stackrel{c}{\hookrightarrow} \F(B_X)\sim \F(X).\]

Finally, since $\F(X)$ is obviously complemented in any $\ell_1$-sum of itself and 
\[\left(\bigoplus_{n\in \N}\left(\bigoplus_{a\in \mathcal{A}}\F(X)\right)_{\ell_1}\right)_{\ell_1} \sim \left(\bigoplus_{a\in \mathcal{A}}\F(X)\right)_{\ell_1},\] 
the conclusion follows from a standard application of Pe{\l}czy\'nski decomposition method.
\end{proof}

\begin{remark}\label{Cardinalell1} It follows from Theorem \ref{ell_1Sum} that if $X$
 is a Banach space of density $\kappa$, then for every cardinal $1\leq \kappa_0\leq \kappa$ holds
\begin{align*}
\F(X)&\sim \left(\bigoplus_\kappa\F(X)\right)_{\ell_1} \sim \left(\bigoplus_{\kappa_0\times \kappa}\F(X)\right)_{\ell_1}\sim \left(\bigoplus_{\kappa_0}\left(\bigoplus_\kappa\F(X)\right)_{\ell_1}\right)_{\ell_1}\sim \left(\bigoplus_{\kappa_0}\F(X)\right)_{\ell_1}.
\end{align*}
\end{remark}


\section{On large $\ell_{\infty}$-sums of Lipschitz function spaces}
\label{Sec-ellinfty}

The following proof was inspired by \cite[Proposition 1]{Johnson}, \cite[Lemma 10]{CanKau} and \cite[Lemma 1.3]{CCD}.

\begin{proof}[Proof of Theorem \ref{UltraLIP}]
Let $T:\Lip(M)\to \left(\bigoplus_{N\in \varGamma}\Lip(N)\right)_{\ell_{\infty}}$ be given by
\[T(f)=(f\restriction_N)_{N\in \varGamma}.\]
It is evident that $T$ is a well defined linear operator with $\|T\|\leq 1$.

On the other hand, for each $N\in \varGamma$ and $g\in \Lip(N)$ we define a map $S_N(g):W\to \R$ by
\begin{displaymath}
S_N(g)(x)= \left\{
\begin{array}{ll}
0 & \text{ if }x\not\in N;\\
g(x) & \text{ if }x\in N.
\end{array} \right.
\end{displaymath}
Let $\mathcal{U}$ be a fixed directed ultrafilter on $\varGamma$, see Remark \ref{Ultralimits}. For each $G=(g_N)_{N\in \varGamma}\in (\bigoplus_{N\in \varGamma}\Lip(N))_{\ell_\infty}$, denoting $\varOmega=\bigcup_{N\in \varGamma}N$, we define the map $G^{\mathcal{U}}:\varOmega\to \R$ by the formula
\[G^{\mathcal{U}}(x)=\lim_{\mathcal{U}}S_N(g_N)(x).\]

Since $\sup_{N\in \varGamma}|S_N(g_N)(x)|\leq \|G\|_{\ell_{\infty}}d(x,0)<\infty$, $G^{\mathcal{U}}(x)$ is a real number for each $x\in \varOmega$. Given $x,y \in \varOmega$, let $Z\in \varGamma$ such that $x,y \in Z$ and
let $(N_n)_{n \in \N}$ be a sequence as in Remark \ref{Ultralimits}, so that $Z\subset N_1\subset N_2 \subset N_3\subset \ldots$ and  
\[\lim_{\mathcal{U}}|S_N(g_N)(x)-S_N(g_N)(y)|= \lim_{n\to \infty}|S_{N_n}(g_{N_n})(x)-S_{N_n}(g_{N_n})(y)|.\]
We have
\begin{align*}
|G^{\mathcal{U}}(x)-G^{\mathcal{U}}(y)|&=\lim_{\mathcal{U}}|S_N(g_N)(x)-\lim_{\mathcal{U}}S_N(g_N)(y)|=\lim_{n\to \infty}|g_{N_n}(x)-g_{N_n}(y)|\\
&\leq \sup_{N\in \varGamma}\|(g_N)_{N\in \varGamma}\|_{\ell_{\infty}}d(x,y)=\|G\|_{\infty}d(x,y).
\end{align*}

We deduce that $G^{\mathcal{U}}$ is a Lipschitz function, and since it is clear that $G^{\mathcal{U}}(0)=0$, we have $G^{\mathcal{U}}\in \Lip(\varOmega)$. Recalling that $\varOmega$ is dense in $M$,  $G^{\mathcal{U}}$ can be uniquely extended to a function $S(G)\in \Lip(M)$.
In this way, we have a well-defined map $S:\left(\bigoplus_{N\in \varGamma}\Lip(N)\right)_{\ell_{\infty}}\to \Lip(M)$. 

We claim that $S$ is linear. Given $G=(g_N)_{N\in \varGamma},\ H=(h_N)_{N\in \varGamma}\in  \left(\bigoplus_{N\in \varGamma}\Lip(N)\right)_{\ell_{\infty}}$ and $\lambda \in \R$, let $x\in \varOmega$ be arbitrary. We pick $Z\in \varGamma$ such that $x\in Z$ and
let $(N_n)_{n \in \N}$ be a sequence as in Remark \ref{Ultralimits}, so that $Z\subset N_1\subset N_2 \subset N_3\subset \ldots$ and  
\small
\[\lim_{\mathcal{U}}(S_N(\lambda g_N+h_N)(x)-\lambda S_N(g_N)(x)-S_N(h_N)(x))= \lim_{n\to \infty}(S_{N_n}(\lambda g_{N_n}+h_{N_n})(x)-\lambda S_N(g_{N_n})(x)-S_N(h_{N_n})(x)).\]\normalsize
Since $S_{N_n}(\lambda g_{N_n}+h_{N_n})(x)-\lambda S_N(g_{N_n})(x)-S_N(h_{N_n})(x)=0$ for each $n \in \N$ and since $\varOmega$ is dense in $M$ we deduce that $S(\lambda G+H)=\lambda S(G)+S(H)$. Hence, $S$ is a bounded linear operator and follows from the construction that $\|S\|\leq 1$.

Finally, let $f\in \Lip(M)$ and $x\in \varOmega$ be arbitrary. Let $Z\in \varGamma$ be such that $x\in Z$ and $(N_n)_n$ be a sequence in $\varGamma$ as in Lemma \ref{Ultralimits}, such that $Z\subset N_1\subset N_2 \subset N_3\subset \ldots$ and
\[\lim_{\mathcal{U}}S_N(f\restriction_N)(x)=\lim_{n\to \infty}S_{N_n}(f\restriction_{N_n})(x).\]
We have
\begin{align*}
S(T(f))(x)=\lim_{n\to \infty}S_{N_n}(f\restriction_{N_n})(x)=\lim_{n\to \infty}f(x)=f(x)
\end{align*}
and deduce that $S\circ T=\Id_{\Lip(M)}$ is the identity operator. Hence $P=T\circ S$ is a projection of norm $1$ of 
$\left(\bigoplus_{N\in \varGamma}\Lip(N)\right)_{\ell_{\infty}}$ onto an isometric copy of $\Lip(M)$.

\end{proof}

The next theorem and its corollary were suggested to the authors by Prof. Marek C\'uth.

\begin{theorem}\label{UniformlyComplementedSequence}
Let $Y$ be an infinite-dimensional Banach space of density $\kappa$ and let $X$ be a Banach space admitting a sequence $(X_n)_n$ of uniformly complemented finite-dimensional subspaces of $X$. Suppose that, for some $\lambda>0$, every finite-dimensional subspace $Y_0$ of $Y$ is contained in a subspace $Y_1$ of $Y$ with $Y_1\stackrel{\lambda}{\sim}X_n$ for some $n \in \N$. Then
\[\Lip(Y)\stackrel{c}{\hookrightarrow}\left(\bigoplus_{\kappa}\Lip(X)\right)_{\ell_{\infty}}.\]
\end{theorem}
\begin{proof}
Let $\varGamma\subset Y$ be a dense subset of cardinality $\kappa$. We define $\varGamma_0=\varGamma$ and for $0<n<\omega$ we let $\varGamma_{n+1}=\{j\subset \varGamma_{n}:j\text{ is finite and non-empty}\}$. Finally we put $\varGamma_{\omega}=\bigcup_{n<\omega}\varGamma_{n}$. 

For each $j\in \varGamma_0$, we fix a subspace $Y_j$ of $Y$ containing $\Span\{j\}$ and such that $Y_j\stackrel{\lambda}{\sim}X_{n_j}$ for some $n_j \in \N$. Given $n<\omega$ arbitrary, if $Y_i$ is defined for all $i\in \varGamma_n$, then for each $j\in \varGamma_{n+1}$ we fix $Y_j$ containing  $\bigoplus_{i\in j} Y_i$ and such that $Y_j\stackrel{\lambda}{\sim}X_{n_j}$ 
for some $n_j \in \N$. By induction we obtain the collection $\varOmega=\{Y_j:j\in \varGamma_{\omega}\}$. It is clear that $\varOmega$ has cardinality $\kappa$ and $\bigcup_{Z\in \varOmega} Z$ is dense in $Y$. We claim that $\varOmega$ is also directed by upward inclusion. Indeed, 
given $j_1,j_2\in \varGamma_{\omega}$ arbitrary, assume that $j_1\in \varGamma_n$ and $j_2\in \varGamma_m$ with $m\geq n$. By our construction, there exists  $j^{\prime}_1\in \varGamma_m$ such that $Y_{j_1} \subset Y_{j_1^{\prime}}$. Then $t=\{j^{\prime}_1,j_2\}\in \varGamma_{m+1}$ and
$Y_t$ contains $Y_{j_1^{\prime}}\oplus Y_{j_2}$ which, in turn, contains $Y_{j_1}$ and $Y_{j_2}$. This establishes our claim.

Finally, from Theorem \ref{UltraLIP}, from the construction of $\varOmega$ and the hypothesis that $(X_n)_n$ is a sequence of uniformly complemented subspaces of $X$, follows 
\[\Lip(Y)\stackrel{c}{\hookrightarrow}\left(\bigoplus_{j\in \varGamma_{\omega}}\Lip(Y_j)\right)_{\ell_{\infty}}\stackrel{\lambda}{\sim}\left(\bigoplus_{j\in \varGamma_{\omega}}\Lip(X_{n_j})\right)_{\ell_{\infty}}\stackrel{c}{\hookrightarrow}\left(\bigoplus_\kappa\Lip(X)\right)_{\ell_{\infty}}.\]
\end{proof}

From Theorem \ref{UniformlyComplementedSequence} the following result is immediate.

\begin{corollary}\label{Lpcomplementation}For some $1\leq p \leq \infty$, Let $X$ be a $\mathcal L_p$-space of density $\kappa$. If $p<\infty$, then
\[\Lip(X)\stackrel{c}{\hookrightarrow}\left(\bigoplus_{\kappa}\Lip(\ell_p)\right)_{\ell_{\infty}}.\]
If $p=\infty$, then 
\[\Lip(X)\stackrel{c}{\hookrightarrow}\left(\bigoplus_{\kappa}\Lip(c_0)\right)_{\ell_{\infty}}.\]
\end{corollary}

We are now ready to demonstrate our third main result.

\begin{proof}[Proof of Theorem \ref{Lpspaces}]
 Let $X$ be a non-separable $\mathcal{L}_p$ space of cardinality $\kappa$ (the separable case was established in \cite[Theorem 3.3]{CCD}). We assume that $1\leq p <\infty$. The case $p=\infty$ is similar. By arguing as in the proof of \cite[Theorem 3.3]{CCD}, we obtain 
that $\Lip(\ell_p)\stackrel{c}{\hookrightarrow} \Lip(X)$. From Theorem \ref{ell_1Sum} follows
\[\left(\bigoplus_{\kappa}\Lip(\ell_p)\right)_{\ell_\infty} \stackrel{c}{\hookrightarrow} \left(\bigoplus_{\kappa}\Lip(X)\right)_{\ell_\infty}\sim \Lip(X).\]

On the other hand, by Corollary \ref{Lpcomplementation}, we have
\[\Lip(X)\stackrel{c}{\hookrightarrow} \left(\bigoplus_{\kappa}\Lip(\ell_p)\right)_{\ell_{\infty}}.\]
An application of Pe{\l}czy\'nski decomposition method yields 
\[\Lip(X)\sim \left(\bigoplus_{\kappa}\Lip(\ell_p)\right)_{\ell_{\infty}}.\]

We are done by repeating the process for $X=\ell_p(\kappa)$.
\end{proof}

From our results and some ideas from \cite[Theorem 3.1]{CCD} we can derive one more theorem. We recall that a Banach space $X$ is said to contain 
$\ell_{\infty}^n$'s uniformly (or $\lambda$-uniformly), if there is $\lambda\geq 1$ such that, for each $n \in \N$, there is a subspace $Y$ of $X$ with $Y\stackrel{\lambda}{\sim} \ell_{\infty}^n$. It is well-known that $X$ contains $\ell_{\infty}^n$'s uniformly if an only if $X$ does not have finite cotype, see \cite[Theorem 14.1]{DHT}

\begin{theorem}Let $X$ be a Banach space of density $\kappa$ containing $\ell_{\infty}^n$'s uniformly. Then, for every Hausdorff compact space $K$ of weight $\kappa$, 
\[\Lip(C(K))\stackrel{c}{\hookrightarrow}\Lip(X).\]
\end{theorem}
\begin{proof}
Since $X$ contains $\ell_{\infty}^n$'s uniformly, by using Hahn-Banach theorem we may deduce that $\ell_\infty^n$'s are uniformly complemented in $X$. From Theorem \ref{UltraLIP} and Remark \ref{Cardinalell1} (or simply \cite[Lemma 1.3]{CCD} and \cite[Theorem 3.1]{Kauf}) we have
\[\Lip(c_0)\stackrel{c}{\hookrightarrow}\left(\bigoplus_{\N}\Lip(\ell_\infty^n)\right)_{\ell_{\infty}}\stackrel{c}{\hookrightarrow}\left(\bigoplus_{\N}\Lip(X)\right)_{\ell_{\infty}}\sim \Lip(X).\]

Let $K$ be an infinite Hausdorff compact space of weight $\kappa$. Since $C(K)$ is a $\mathcal{L}_\infty$-space \cite[p. 198-199]{LinTza}, from Corollary \ref{Lpcomplementation}, the previous relation and Theorem \ref{ell_1Sum} follows 
\[C(K)\stackrel{c}{\hookrightarrow} \left(\bigoplus_{\kappa}\Lip(c_0)\right)_{\ell_{\infty}} \stackrel{c}{\hookrightarrow} \left(\bigoplus_{\kappa}\Lip(X)\right)_{\ell_{\infty}} \sim \Lip(X).\]
\end{proof}

\section{Acknowledgments}

The authors wish to thank Prof. Marek C\'uth for valuable suggestions that improved the original manuscript. The first-named author was supported by Funda\c c\~ao de Amparo \`a Pesquisa do Estado de S\~ao Paulo - FAPESP No. 2016/25574-8. The second-named author was supported by Coordena\c c\~ao de Aperfei\c coamento de Pessoal de N\'ivel Superior - CAPES.

\bigskip

\bibliographystyle{plain}

\end{document}